\documentclass[a4paper,10pt,reqno]{amsart}

\usepackage{mathtools,amssymb,microtype}
\usepackage[hidelinks]{hyperref}

\AddToHook{bfseries}{\boldmath}

\theoremstyle{plain}
\newtheorem*{theorem*}{Theorem}
\newtheorem{theorem}{Theorem}[section]
\newtheorem{proposition}[theorem]{Proposition}
\newtheorem{lemma}[theorem]{Lemma}
\newtheorem{corollary}[theorem]{Corollary}

\theoremstyle{definition}

\theoremstyle{remark}
\newtheorem{remark}[theorem]{Remark}
\newtheorem{example}[theorem]{Example}

\numberwithin{equation}{section}

\DeclareMathOperator{\adj}{adj}
\DeclareMathOperator{\Hess}{Hess}
\DeclareMathOperator{\Tr}{Tr}

\newcommand{\Lie}[1]{\mathrm{#1}}
\newcommand{\G}{\Lie{G}}
\newcommand{\GL}{\Lie{GL}}
\newcommand{\SO}{\Lie{SO}}
\newcommand{\Spin}{\Lie{Spin}}
\newcommand{\Sp}{\Lie{Sp}}
\newcommand{\SU}{\Lie{SU}}
\newcommand{\Un}{\Lie{U}}

\newcommand{\bC}{\mathbb{C}}
\newcommand{\bR}{\mathbb{R}}
\newcommand{\bZ}{\mathbb{Z}}

\newcommand{\vol}{\mathrm{vol}}
\newcommand{\Id}{\mathrm{Id}}

\newcommand{\Hd}[1]{{\star}_{#1}#1}

\newcommand{\hook}{\mathbin{\lrcorner}}

\newcommand{\eqbreak}[1][2]{\\ &\hspace{#1em}}

\title{Nearly Parallel \texorpdfstring{$\G_{2}$}{G2}-Structures with Torus Symmetry}

\author{Giovanni Russo}
\address{Mathematics Area\\
SISSA - Scuola Internazionale Superiore di Studi Avanzati\\
via Bonomea 265\\
34136 Trieste (TS)\\
Italy} %
\email{girusso@sissa.it}

\author{Andrew Swann}
\address{Department of Mathematics and DIGIT\\
Aarhus University\\
Ny Munkegade 118\\
Bldg 1530\\
DK-8000 Aarhus C\\
Denmark} %
\email{swann@math.au.dk}

\begin{document}

\begin{abstract}
  We study nearly parallel $\G_{2}$-structures with a three-torus symmetry via
  multi-moment map techniques.
  An effective three-torus action on a nearly parallel $\G_{2}$-manifold yields
  a multi-moment map.
  The torus acts freely on its regular level sets, so they are torus bundles
  over smooth three-dimensional manifolds.
  We show that the geometry of the base spaces is specified by two triples of
  closed two-forms related by a Riemannian metric.
  We then describe an inverse construction producing invariant nearly parallel
  $\G_{2}$-structures from three-dimensional data.
  We observe that locally this may produce examples with four-torus symmetry.
\end{abstract}

\subjclass{Primary 53C25; Secondary 53C26, 53D20, 57S25}

\maketitle

\section*{Introduction}

A $\G_{2}$-structure on a seven-dimensional manifold $M$ is a reduction of the
principal frame bundle to the compact, connected, simply connected, simple Lie
group~$\G_{2}$.
This is equivalent to having a global three-form $\varphi$ on $M$ that is
pointwise linearly equivalent to a standard model on~$\bR^7$ with the cross
product defined by multiplication of imaginary octonions \cite{bryant}.
In particular the three-form $\varphi$ induces a Riemannian metric $g$ and an
orientation on~$M$, and hence, a Hodge star operator~$\star$.

In this paper, we are interested in nearly parallel $\G_{2}$-structures, which
are the ones for which the three-form $\varphi$ satisfies the condition
$d\varphi = \lambda\, {\star \varphi}$, for a positive constant $\lambda$
\cite{gray}.
For such structures the induced Riemannian metric is Einstein with positive
scalar curvature \cite{besse, bryant1}, so complete examples are necessarily
compact, and the geometries are equivalently described as seven-manifolds with a
non-parallel real Killing spinor \cite{baum, friedrich-kms}.
By \cite{bar}, the holonomy of the Riemannian eight-dimensional cone over a
nearly parallel $\G_{2}$ manifold may be (a)~the full group $\Spin(7)$, or one
of the subgroups (b)~$\SU(4)$, (c)~$\Sp(2)$, (d)~${1}$, according to the
dimension of the space of Killing spinors.
As will be discussed below, examples of nearly parallel $\G_{2}$-structures
exist in abundance.
However, case (d) only occurs when the metric is locally isometric to the round
seven-dimensional sphere.

In this paper we focus on nearly parallel $\G_{2}$-structures with torus
symmetry.
Several examples of compact nearly parallel $\G_{2}$-manifolds with large
symmetry groups have been classified.
The homogeneous examples are to be found in \cite{friedrich-kms}, which includes
tables indicating the full isometry group.
In general, this is larger than the group of automorphisms of the
$\G_{2}$-structure, as there can be isometries that do not preserve the
three-form.
However, Reidegeld \cite[Theorem~1, part 2, and Table~2]{reidegeld} summarises
the automorphism groups, and one sees that the rank of the symmetry group in the
homogeneous examples is at most~$3$.
These include the remarkable family of Aloff-Wallach
spaces~$\SU(3)/T^{1}_{k,\ell}$, described as simply-connected spaces of the form
$\Un(3)/T^{2}_{a,b}$, with $T^{2}_{a,b} \cap \SU(3)$ one-dimensional and
$T^{2}_{a,b} \cap \bC \Id_{3}$ zero-dimensional.
We therefore focus on geometries with $T^{3}$-symmetry.

The general Aloff-Wallach spaces are of class (a) above, and this family
contains infinitely many distinct homotopy types \cite{kreck-stolz} and they are
all topologically formal~\cite{fernandez-fkm}.
Geometries in class (b) are Sasaki-Einstein, their cone is Calabi-Yau.
Infinitely many compact examples of these Sasaki-Einstein structures with
$T^{3}$-symmetry may be constructed starting from ordinary toric geometry
\cite{cho-fo,futaki-ow}, see also \cite{cvetic}.
Geometries from class (c) are also $3$-Sasaki manifolds \cite{boyer-galicki}.
Here for each second Betti number $b_{2} > 0$ there are infinitely many distinct
compact homotopy types where the $\G_{2}$-structure has $T^{3}$-symmetry
\cite{boyer-galicki,hepworth}.
The examples also give structures in class~(a) via a canonical variation
\cite{friedrich-kms}, and these new structures still have $T^{3}$-symmetry, and
in fact at least $T^{2} \times \SU(2)$-symmetry.
The general theories of Sasaki-Einstein and $3$-Sasaki manifolds, ensure that
the rank of the symmetry group of the $\G_{2}$-structure cannot be higher than
three in cases (b) and~(c), which again justifies the choice of focussing on
$T^{3}$-symmetry.

For a general nearly parallel $\G_{2}$-manifold~$M$ with $T^{3}$-symmetry, we
will show how to construct a multi-moment $\nu\colon M \to \bR$ in the sense of
\cite{madsen-swann}.
This leads us to investigate the geometry of the quotients
$Q = \nu^{-1}(s)/T^{3}$ of the level sets of~$\nu$, when $s$ is a non-zero, regular
value.
These~$Q$ are smooth three-manifolds which carry a surprising amount of
geometric data.
As $\nu^{-1}(s)$ is a $T^{3}$-bundle over~$Q$, there is, by the usual Chern-Weil
theory, a cohomology class in the integral lattice of
$H^{2}(Q,\bZ^{3}) \otimes \bR = H^{2}(Q,\bR^{3})$ describing this bundle via its
curvature.
However, in our situation when $Q$ comes from a $\G_{2}$-structure as above, we
show that the class has a representative $\tau/s$,
$\tau= (\tau_{1},\tau_{2},\tau_{3})^{T} \in \Omega^{2}(Q,\bR^{3})$, for which the three-components
are pointwise linearly independent closed two-forms, even when $b_{2}(Q) = 0$.
Additionally, $Q$ has a Riemannian metric represented by
$H \in C^{\infty}(Q,M_{3}(\bR))$, such that $\sigma \coloneqq H\tau$ is also a closed
$\bR^{3}$-valued two-form.
At a given $s$, we show that the nearly parallel $\G_{2}$-structure with
$T^{3}$-symmetry can be recovered from the data $\tau$ and $H$, or equivalently
$\tau$ and~$\sigma$, under a particular geometric flow.
As a consequence of this inverse construction, we are able to assert the local
existence of nearly parallel $\G_{2}$-structures with $T^{4}$-symmetry, rather
than just $T^{3}$.

Nearly parallel $\G_{2}$-structures share some features with nearly K\"ahler
manifolds in dimension six \cite{gray-nk,gray}.
Indeed, the two geometries are the only exceptional cases in the classification
\cite[Theorem 5.14]{cleyton-swann-E} of geometries with metric connection whose
torsion is parallel and whose holonomy representation is irreducible.
Both geometries are Einstein of positive scalar curvature, and their cones have
special Riemannian holonomy \cite{gray-nk,russo,bar,bryant}.
A significant difference is that in the nearly K\"ahler case, there are currently
very few complete examples \cite{butruille,foscolo-haskins}.
In \cite{russo2, russo-swann} we studied nearly K\"ahler examples with
$T^{2}$-symmetry.
Again there is a multi-moment map and quotients of regular level sets are
three-manifolds with geometric structure, but it was not at all clear how that
structure would change when passing to the case of nearly parallel
$\G_{2}$-structures with $T^{3}$-symmetry.
The $\G_{2}$-case also turns to have a technical difficulty, in that it is not
immediately apparent that the flow equations preserve the symmetry of the matrix
$H$ determining the Riemannian structure on~$Q$.

The nearly K\"ahler case admits examples of larger symmetry rank, namely the
three-symmetric metric on $S^{3} \times S^{3} = \SU(2)^{3}/\SU(2)_{\Delta}$ has
$T^{3}$-symmetry.
Nearly K\"ahler six-manifolds with $T^{3}$-symmetry have been described locally by
Mororianu \& Nagy \cite{mororianu-nagy}, but at this time no further complete
examples have been found.
Similarly, in the nearly parallel $\G_{2}$-case, no complete examples with the
larger $T^{4}$-symmetry are currently known.

The structure of the paper is the following.
In Section~\ref{sec:torus-symmetry-and-multi-moment-map} we construct the
multi-moment map~$\nu$, prove that $\nu$ always has regular values and that the
three-torus $T^3$ acts freely on the corresponding level sets of~$\nu$.
We see that if $p$ is a critical point with $\nu(p)$ non-zero, then the
$T^{3}$-orbit through~$p$ is an associative and minimal submanifold.
As an example of the set-up, we consider structures on the round seven-sphere,
showing that the quotient $Q = \nu^{-1}(s)/T^{3}$, for $s$ any regular value, is
smoothly the three-sphere~$S^{3}$.
We also note that results of \cite{kawai} for the squashed metric on the
seven-sphere lead to the same conclusion.
We also give a brief description of the relations to Sasaki-Einstein and
$3$-Sasaki geometry in the cases (b) and (c) above in
Section~\ref{sec:special-classes}.
For a regular value $s$ of~$\nu$, we determine the induced geometry on the
three-manifold $Q = \nu^{-1}(s)/T^{3}$ in Section \ref{sec:torus-reduction}.
Section \ref{sec:inverse-construction} is devoted to an inverse construction.
Finally in Section~\ref{sec:exampl-consequences}, we discuss several local
constructions of nearly parallel $\G_{2}$-structures, a couple of particular
examples and the assertion that examples with $T^{4}$-symmetry exist locally.
Indeed we construct an incomplete example with a cohomogeneity-one action of
$T^{3} \times \SU(2)$.
Such examples have not been seen in previous discussion of cohomogeneity-one
examples \cite{cleyton-swann,podesta}.

\smallbreak \textbf{Acknowledgements.} The first named author is partially supported
by INdAM-GNSAGA, and by the PRIN 2022 Project (2022K53E57) -- PE1 -- \lq\lq Optimal
Transport: new challenges across analysis and geometry\rq\rq.
We thank Thomas Bruun Madsen for useful conversations on topics related to this
paper.

\section{Torus symmetry and multi-moment map}
\label{sec:torus-symmetry-and-multi-moment-map}

Let $M$ be a connected manifold of dimension seven.
A $\G_2$-structure on $M$ is a three-form $\varphi$ that induces a Riemannian metric
$g$ and volume form on $M$ via the formula
\begin{equation}
  \label{eq:g-vol}
  6g(X,Y)\vol \coloneqq (X \hook \varphi) \wedge (Y \hook \varphi) \wedge \varphi.
\end{equation}
We let $\star$ be the Hodge star operator induced by $g$ and the volume form.
Pointwise, $\varphi$ and $\star \varphi$ can be written in terms of a coframe
$e^1,\dots, e^7$ as
\begin{align}
  \label{eq:np-str1}
  \varphi & = e^{123}-e^1(e^{45}+e^{67})-e^2(e^{46}+e^{75})-e^3(e^{47}+e^{56}), \\
  \label{eq:np-str2}
  \star \varphi & = e^{4567}-e^{23}(e^{45}+e^{67})-e^{31}(e^{46}+e^{75})-e^{12}(e^{47}+e^{56}),
\end{align}
where the shorthand $e^{ijk}$ stands for $e^i \wedge e^j \wedge e^k$, etc.
We will use the same notational convention on different forms.
Notice that at each point $\vol = e^{12\dots 7}$,
$\varphi \wedge {\star \varphi} = 7\vol$, and $g=\sum_{k=1}^7 (e^k)^2$.
We recall that $\varphi$ and $g$ define a cross-product via the identity
\begin{equation*}
  g(X \times Y,Z) \coloneqq \varphi(X,Y,Z) .
\end{equation*}
A three-plane in $T_{x}M$ is \emph{associative} if it is closed under the
cross-product, see \cite[IV.1.A. Definition~1.3]{harvey-lawson}.

We will denote by $E_1,\dots,E_7$ the dual vectors of $e^1,\dots,e^7$.
The condition that $(M,\varphi)$ is a nearly parallel $\G_{2}$-manifold is that
$d\varphi$ is proportional to $\star \varphi$.
Up to a homothetic scaling, we can arrange that the nearly parallel condition
reads
\begin{equation*}
  d\varphi = 4\,{\star \varphi},
\end{equation*}
cf.\ \cite[p.~567]{bryant}.

Let $T^3$ be a three-dimensional torus acting on $(M,\varphi)$ effectively
preserving $\varphi$.
Let $U_1$, $U_2$, $U_3$ be the vector fields generated by the torus-action on
$M$.  The map
\begin{equation*}
  \nu \coloneqq \varphi(U_1,U_2,U_3)
\end{equation*}
is a multi-moment map, cf.\ Madsen-Swann \cite{madsen-swann}.
Indeed, $\nu$ is $T^3$-invariant and satisfies
\begin{equation*}
  d\nu = -4\,{\star \varphi}(U_1,U_2,U_3,{}\cdot{}).
\end{equation*}
The latter identity follows by Cartan's formula, the fact that the Lie
derivatives $\mathcal{L}_{U_i}\varphi$ vanish for each $i=1,2,3$, and
$[U_i,U_j]=0$ for all $i,j$.

Recall that $\G_{2}$ acts transitively on the Stiefel manifold of two-frames in
$\bR^7$ \cite{bryant}, and acts linearly on each tangent space of $M$ preserving
$\varphi$.
Note that $E_3 = E_1 \times E_2$, and that the subgroup of $\G_{2}$ preserving
$E_{1},E_{2}$, and hence $E_{3}$, is isomorphic to $\SU(2)$ that acts
transitively on the unit vectors orthogonal to $E_{1},E_{2},E_{3}$.
Thus locally, up to the action of $\G_{2}$, we can then take $E_1,\dots,E_7$
such that
\begin{align*}
  \begin{pmatrix}
    U_{1}\\
    U_{2}\\
    U_{3}
  \end{pmatrix}
  = V
  \begin{pmatrix}
    E_{1}\\
    E_{2}\\
    E_{3}\\
    E_{4}
  \end{pmatrix}
  ,\qquad\text{where}\quad
  V =
  \begin{pmatrix}
    p&0&0&0\\
    q_{1}&q_{2}&0&0\\
    r_{1}&r_{2}&r_{3}&r_{4}
  \end{pmatrix}
\end{align*}
for some functions $p$, $q_i$, $r_j$.  For such a basis, we have
\begin{equation}
  \label{eq:nu-pqr-dnu}
  \nu = pq_2r_3
  \quad\text{and}\quad
  d\nu = 4pq_2r_4e^7.
\end{equation}

\begin{proposition}
  \label{prop:critical}
  Critical points of $\nu$ may be characterised as follows:
  \begin{enumerate}
  \item\label{it:crit-zero} $x \in \nu^{-1}(0)$ is a critical point for $\nu$ if
    and only if $U_1,U_2,U_3$ are linearly dependent at $x$.
  \item\label{it:crit-nonzero} $x \in M \setminus \nu^{-1}(0)$ is a critical
    point for $\nu$ if and only if $\mathrm{Span}_{x}\{U_1,U_2,U_3\}$ is an
    associative three-plane.
  \end{enumerate}
  Thus $T^{3}$-orbits of critical points~$x$ with $\nu(x) \ne 0$ are
  associative, hence minimal, submanifolds.
\end{proposition}

\begin{proof}
  In the following all expressions are at a fixed $x \in M$, and we use the
  local representation of $U_{i}$ above.

  For (\ref{it:crit-zero}), suppose $\nu(x)=0$ and $d\nu_{x}=0$. Then $pq_2=0$
  or $r_3^2+r_4^2=0$.
  It is then clear by the above pointwise expressions that $U_1,U_2,U_3$ are
  linearly dependent.
  The converse is obvious by the global expressions for $\nu$ and $d\nu$.

  Next, for (\ref{it:crit-nonzero}), assume $x$ is a critical point for $\nu$
  but $\nu(x) \ne 0$.
  We have that $pq_2r_{3} \neq 0$, $r_4=0$, and the vectors $U_1,U_2,U_3$ are
  linearly independent, by part~(\ref{it:crit-zero}).
  Thus the span of $U_{1},U_{2},U_{3}$ is that of $E_{1},E_{2},E_{3}$, which is
  associative.

  Conversely, if the span is three-dimensional and associative, then the local
  form gives that $U_{3}$ is a non-zero linear combination
  $U_3 = aU_1+bU_2+c\,U_1 \times U_2$, with $c \neq 0$.
  Thus $r_4=0$, so $d\nu_{x} = 0$ and $x$~is critical.
  Also,
  $\nu(x) = c\varphi_x(U_1,U_2,U_1 \times U_2) = c\lVert U_1 \times U_2\rVert^2
  \neq 0$, as $U_1$ and $U_2$ are linearly independent.

  Minimality of associative submanifolds of nearly parallel $\G_{2}$-manifolds
  was noted in \cite[Corollary~3.6]{lotay} and~\cite{ball-madnick}, for example:
  a submanifold $N \subset M$ is associative if and only if the cone
  $C(N) = \bR_{>0} \times N \subset (C(M), \overline{g} = dt^{2} + t^{2}g)$ is
  calibrated for the cone $\Spin(7)$-structure given by
  $\Psi = dt \wedge \varphi + {\star \varphi}$.
  Thus $C(N)$ is minimal in the cone by \cite{harvey-lawson}.
  But $C(N) \subset C(M)$ and $N \subset M$ have identical mean curvature
  vectors on $N = \{1\} \times N$, since
  $\overline{\nabla}_{\partial_{t}}\partial_{t} = 0$ and
  $(\overline{\nabla}_{X}X)^{\perp C(M)} = (\nabla_{X}X +
  (1/t)g(X,X)\partial_{t})^{\perp C(M)}_{t=1} = (\nabla_{X}X)^{\perp M}$ for
  $X \in N$ extended by $X_{(t,p)} = X_{p}$.
\end{proof}

Let us write $H_{ij} = g(U_i,U_j)$ and $H = (H_{ij})_{i,j=1,2,3}$.
So $H = VV^{T}$, and
\begin{equation*}
  h \coloneqq \det H = p^{2}q_{2}^{2}\tilde{r}^{2},
\end{equation*}
for
\begin{equation*}
  \tilde{r}^{2} = r_{3}^{2} + r_{4}^{2}.
\end{equation*}
From \eqref{eq:nu-pqr-dnu}, we get
$\det H = \nu^2+\frac{1}{16}\lvert d\nu\rvert^2$, or equivalently
\begin{equation*}
  \rho \coloneqq h - \nu^{2} = \frac{1}{16}\lvert d\nu\rvert^{2}.
\end{equation*}
This implies that $\nu$ cannot be identically zero.
If it were, then $h = \det H \equiv 0$ and the vectors $U_1$, $U_2$, $U_3$ would
be linearly dependent at each point; but this contradicts the effectiveness of
the $T^{3}$-action.  We show later that $\nu$ cannot be constant in general.

Let $\theta_i$ be the dual one-form to $U_i$, and write
$\theta = (\theta_{1},\theta_{2},\theta_{3})^{T}$,
$U^{\flat} = (U_{1}^{\flat}, U_{2}^{\flat}, U_{3}^{\flat})^{T}$.  Then
\begin{align*}
  H\theta = U^{\flat}.
\end{align*}
We define $\alpha = (\alpha_{1},\alpha_{2},\alpha_{3})^{T}$, a vector of basic
$T^{3}$-invariant one-forms, by
\begin{equation}
  \label{eq:horizontal-forms}
  \alpha_i \coloneqq \nu \theta_i - \varphi(U_j,U_k,{}\cdot{})
\end{equation}
for $(ijk)=(123)$ as permutations.
A calculation yields the following pointwise expressions
\begin{equation}
  \label{eq:local-theta-alpha}
  \theta = X
  \begin{pmatrix}
    e^{1}\\
    e^{2}\\
    e^{3}\\
    e^{4}
  \end{pmatrix},\quad
  \alpha = pq_{2}r_{4}\,X
  \begin{pmatrix}
    e^{6}\\
    -e^{5}\\
    e^{4}\\
    -e^{3}
  \end{pmatrix},
\end{equation}
where
\begin{equation*}
  X = \frac{1}{pq_{2}\tilde{r}^{2}}
  \begin{pmatrix}
    q_{2} & - q_{1} & q_{1}r_{2} - q_{2}r_{1}  \\
    0 & p & - pr_{2}  \\
    0 & 0 & pq_{2}
  \end{pmatrix}
  \begin{pmatrix}
    \tilde{r}^{2} & 0 & 0 & 0 \\
    0 & \tilde{r}^{2} & 0 & 0 \\
    0 & 0 & r_{3} & r_{4}
  \end{pmatrix}
  .
\end{equation*}
As $X$ has rank~$3$, we see that away from the critical points of~$\nu$, the
seven one-forms $d\nu$, $\theta_{i}$, $\alpha_{j}$, $i,j=1,2,3$, constitute a
basis of the cotangent space of~$M$.

\begin{proposition}
  The image of the multi-moment map $\nu \colon M \to \bR$ is an interval with
  non-empty interior.
  If $M$ is complete, then $\nu(M) = [a,b]$ is a compact interval with $a<b$.
\end{proposition}

\begin{proof}
  For the first point, since $M$ is connected, it is enough to show that $\nu$
  is not constant on $M$.
  A computation using the pointwise expressions of the forms gives
  \begin{equation*}
    \star d\nu = \frac{4h}{\rho}\theta_{123} \wedge \alpha_{123}.
  \end{equation*}
  In particular, if $\nu$ were constant, then $h = \det H$ should vanish
  identically, and hence $U_1$, $U_2$, $U_3$ would be linearly dependent at each
  point of $M$.
  This contradicts the effectiveness of the torus-action.
  When $M$ is complete it is compact, by Myers' Theorem, so $\nu(M)$ is also
  compact.
\end{proof}

\begin{proposition}
  The multi-moment map $\nu$ has regular values.

  For any regular value $s$, the $T^3$-action on $\nu^{-1}(s)$ is free, so
  $\nu^{-1}(s) \to \nu^{-1}(s)/T^3$ is a principal torus-bundle over a
  three-dimensional smooth manifold.
\end{proposition}

\begin{proof}
  By Sard's Theorem, the set of critical values has zero Lebesgue measure
  in~$\bR$, so there are infinitely many regular values for $\nu$ in
  $(a,b) \neq \varnothing$.

  Let us fix a regular value $s$, and let $x$ be any point in the level set
  $\nu^{-1}(s)$.
  Then $d\nu_x \neq 0$, so $U_1$, $U_2$, $U_3$ are linearly independent at $x$,
  and the torus-action is locally free.

  We show that the action is free.
  If $K$ is the stabiliser in $T^3$ of a point $x \in \nu^{-1}(s)$, then it
  preserves $d\nu$, $\theta_i$, and $\alpha_j$ for all $i,j=1,2,3$, so all of
  $T_xM$.
  By the equivariant tubular neighbourhood theorem, the set
  $A \coloneqq \{y \in M: ky=y, dk_y=\mathrm{Id}_{T_yM} \text{ for all } k \in
  K\}$ is open and closed in $M$.
  Since $M$ is connected and $A$ is non-empty, $A=M$.
  But then $K$ must be trivial, otherwise the torus-action on $M$ would not be
  effective.
\end{proof}

\subsection{The seven sphere}

We consider first the round seven-sphere $S^{7} \subset V = \bR^{8}$ as a nearly
parallel $\G_{2}$-structure, cf.\ \cite{alexandrov,lotay}.
Let $(x^{0},\dots,x^{7})$ be the standard Euclidean coordinates,
$\partial_{i}$~the corresponding vector fields, and consider
$\bR^{7} \subset \bR^{8}$ spanned by $x^{1},\dots,x^{7}$.
On $\bR^{7}$, we have the constant coefficient $\G_{2}$-form $\varphi_{0}$ and
its dual ${\star}_{\bR^{7}}\varphi_{0}$ given by \eqref{eq:np-str1} and
\eqref{eq:np-str2} with $e^{i} = dx^{i}$.
Defining
$\Psi \coloneqq dx^{0} \wedge \varphi_{0} + {\star}_{\bR^{7}}\varphi_{0}$, we
get a four-form on~$\bR^{8}$, whose stabiliser under the action of $\GL(8,\bR)$
is $\Spin(7) \subset \SO(8)$.  We get $S^{7}=\Spin(7)/\G_{2}$.

Note that the Euler vector field $E = \sum_{k=0}^7 x^k\partial_k$ on $\bR^{8}$
restricts to a unit normal~$N$ of $S^{7}$ and $\mathcal{L}_{E}x^{i} = x^{i}$ for
each~$i$.
As $\Psi$ has constant coefficients, we get $\mathcal{L}_{E}\Psi = 4\Psi$.
Putting
\begin{equation*}
  \varphi \coloneqq \iota^*(E \hook \Psi), \qquad \psi \coloneqq \iota^*\Psi,
\end{equation*}
we get from $d\Psi = 0$ that
$d\varphi = \iota^{*}(d (E\hook \Psi)+E \hook d\Psi) = \iota^{*}(\mathcal{L}_{E}\Psi) = \iota^{*}(4\Psi) = 4\psi$.
But $\Spin(7)$ acts by isometries on~$S^{7}$, and at $(1,0,\dots,0)$,
$\psi = {\star}_{\bR^{7}}\varphi_{0} = {\star}\varphi$.
Thus $\psi$ is the Hodge dual of~$\varphi$, and so $\varphi$ equips $S^{7}$ with
a nearly parallel $\G_{2}$-structure for which the induced metric is the round
one.

On $\bR^{8}$ we have four commuting circle actions generated by
\begin{equation*}
  V_i \coloneqq -x^{2i+1}\partial_{2i}+x^{2i}\partial_{2i+1},\quad\text{for $i=0,1,2,3$.}
\end{equation*}
Write
\begin{equation*}
  \Psi = dx^{0123}-\gamma_1\wedge \beta_5-\gamma_2\wedge \beta_6-\gamma_3\wedge \beta_7+dx^{4567},
\end{equation*}
where $\gamma_1 = dx^{01}+dx^{23}$, $\gamma_2=dx^{02}+dx^{31}$,
$\gamma_3=dx^{03}+dx^{12}$, and $\beta_5=dx^{45}+dx^{67}$,
$\beta_6=dx^{46}+dx^{75}$, $\beta_7=dx^{47}+dx^{56}$.  Then, we have
\begin{equation*}
  \mathcal{L}_{V_i}\Psi = -\mathcal{L}_{V_i}(\gamma_2\wedge \beta_6+\gamma_3\wedge \beta_7) = \varepsilon_{i}(\gamma_3 \wedge \beta_6-\gamma_2\wedge \beta_7), \qquad
  i=0,1,2,3,
\end{equation*}
where $\varepsilon_{0} = -1 = \varepsilon_{1}$,
$\varepsilon_{2} = +1 = \varepsilon_{3}$.  So the three vector fields
\begin{equation*}
  U_1 \coloneqq V_0+V_3,\quad U_2 \coloneqq V_1+V_3,\quad U_3 \coloneqq V_2-V_3
\end{equation*}
preserve $\Psi$.
They generate the maximal torus~$T^{3}$ of $\Spin(7)$.
The action then preserves the nearly parallel $\G_{2}$-structure $\varphi$
on~$S^7$, and we get a multi-moment map~$\nu$.  A direct computation gives
\begin{align*}
  \nu & = \varphi(U_1,U_2,U_3) \\
      & = \iota^{*}(\Psi(E,V_0+V_3,V_1,V_2)-\Psi(E,V_0,V_1+V_2,V_3)) \\
      & = 4((x^0x^2-x^1x^3)(x^4x^7+x^5x^6)+(x^1x^2+x^0x^3)(x^5x^7-x^4x^6)).
\end{align*}
Note that writing $z^1=x^0-ix^1$, $z^2=x^2-ix^3$, $z^3=x^4+ix^5$,
$z^4=x^6+ix^7$, we have
\begin{equation}
  \label{eq:nu-S7}
  \nu = 4\operatorname{Im}(z^{1}z^{2}z^{3}z^{4}).
\end{equation}

\begin{proposition}
  \label{prop:round-S7}
  For the above nearly parallel $\G_{2}$-structure on the round seven-sphere,
  the image of the multi-moment map~\eqref{eq:nu-S7} is
  $\nu(S^{7}) = [-1/4,1/4]$.

  The only non-zero critical values are the maximum and minimum $\pm1/4$, and
  the corresponding critical sets are diffeomorphic to~$T^{3}$.

  The $T^{3}$-quotient of the critical points in $\nu^{-1}(0)$ form a trivalent
  graph with four vertices and six edges connecting each pair of distinct
  vertices.

  Each $s \in (-1/4,0) \cup (0,1/4)$ is a regular value of~$\nu$, and
  $\nu^{-1}(s)/T^{3}$ is diffeomorphic to a three-sphere.
\end{proposition}

\begin{proof}
  Using the action of~$T^{3}$, we move any given point to one with
  $x^{1} = 0 = x^{3} = x^{5}$ and $x^{0} \ge 0$, $x^{2} \ge 0$, $x^{4} \ge 0$.
  At points with $x^{1} = 0 = x^{3} = x^{5}$, we have in $\bR^{8}$ that
  \begin{equation*}
    \begin{split}
      d\nu
      &= 4(x^{2}x^{4}x^{7}dx^{0} + x^{0}x^{4}x^{7}dx^{2}
        + x^{0}x^{2}x^{7}dx^{4} \eqbreak
        + x^{0}x^{2}x^{4}dx^{7} + x^{0}x^{2}x^{6}dx^{5}
        - x^{2}x^{4}x^{6}dx^{1} -  x^{0}x^{4}x^{6}dx^{3}).
    \end{split}
  \end{equation*}
  Critical points for $\nu$ on~$S^{7}$ in the slice $x^{1} = 0 = x^{3} = x^{5}$,
  have $d\nu$ proportional to
  $\sum_{i=0}^{7}x^{i}dx^{i} = \sum_{i=0,2,4,6,7}x^{i}dx^{i}$.  So
  \begin{equation}
    \label{eq:critical}
    \begin{gathered}
      x^{2}x^{4}x^{6} = 0,\quad
      x^{0}x^{4}x^{6} = 0,\quad
      x^{0}x^{2}x^{6} = 0,\quad
      0 = \lambda x^{6}, \\
      x^{2}x^{4}x^{7} = \lambda x^{0},\quad
      x^{0}x^{4}x^{7} = \lambda x^{2},\quad
      x^{0}x^{2}x^{7} = \lambda x^{4},\quad
      x^{0}x^{2}x^{4} = \lambda x^{7}.
    \end{gathered}
  \end{equation}
  If $\nu\ne0$, then $x^{j} \ne 0$ for $j = 0,2,4,7$.
  It follows that $\lambda \ne 0$ and thus $x^{6} = 0$.
  Also, we have
  $(x^{0})^{2} = x^{0}x^{2}x^{4}x^{7}/\lambda = (x^{2})^{2} = (x^{4})^{2} =
  (x^{7})^{2}$.
  But $\sum_{j=0,2,4,7} (x^{j})^{2} = 1$, so each $(x^{j})^{2}$ is $1/4$ and
  $\nu = \pm1/4$.
  Thus these are the maximum and minimum values of~$\nu$ and the image is as
  claimed.
  As we can rotate to $x^{k} \geq 0$, $k=0,2,4$, we get $x^{k} = 1/2$ for these
  $k$, and $x^{7} = \pm1/2$.
  Thus the $T^{3}$-orbit has a unique representative, $T^{3}$~acts freely at
  these points and these two critical sets are copies of~$T^{3}$.

  For $\nu=0$, the equations~\eqref{eq:critical} give that in each of the
  following sets $\{x^{0},x^{2},x^{4},x^{6}\}$ and $\{x^{0},x^{2},x^{4},x^{7}\}$
  two variables are zero.
  Thus in $S^{7} \subset \mathbb{C}^{4}$, the critical sets are three-spheres
  that are the intersection of $S^{7}$ with complex planes in $\mathbb{C}^{4}$
  of the form $z^{i} = 0 = z^{j}$, $i \ne j$.
  Any two such three spheres meet in a circle $C_{i}$ given by
  $\lvert z^{i}\rvert =1$, $z^{k} = 0$, for $k\ne i$.
  The $T^{3}$-quotient of the critical points in $\nu^{-1}(0)$, thus form a
  graph with vertices corresponding to the circles~$C_{i}$, and edges
  corresponding to the three spheres.  Thus the graph has the form claimed.

  Let $q = (1/2,0,1/2,0,1/2,0,0,\epsilon/2)$, $\epsilon \in \{\pm1\}$ be our
  preferred point in~$\nu^{-1}(\epsilon/4)$.
  Then $q$ is a critical point for~$\nu$, and we may compute the Hessian $\Hess$
  of~$\nu$ at~$q$ as follows: for $i=0,\dots,7$, let
  $X_{i} = \partial_{i} - x^{i}E$, which is the projection of $\partial_{i}$ to
  the tangent bundle~$TS^{7}$ of~$S^{7}$.
  At a general point, the Hessian of~$\nu$ is given by
  $\Hess(X_{i},X_{j}) = X_{i}(X_{j}\nu) - (\nabla_{X_{i}}X_{j})\nu$, where
  $\nabla$ is the Levi-Civita connection on~$S^{7}$.
  At a critical point, such as~$q$, the last term is zero, since $Y_{q}\nu = 0$
  for all tangent vectors~$Y$.
  Thus $\Hess(X_{i},X_{j})_{q} = X_{i}(X_{j}\nu)$.
  A direct computation gives that the matrix of~$\Hess_{q}$ with respect to the
  spanning set
  $X_{0},X_{2},X_{4},\epsilon X_{7},X_{1},X_{3},-X_{5},\epsilon X_{6}$ is
  \begin{equation*}
    \Hess_{q} =
    \epsilon\begin{pmatrix}
      H_{1} & 0 \\
      0 & H_{2}
    \end{pmatrix}
    ,\qquad\text{with}\quad
    H_{1} = \frac{1}{2}(1_{4\times4} - 3 \Id_{4}),\quad
    H_{2} = - 1_{4\times4},
  \end{equation*}
  where $1_{n\times m}$ is the $(n\times m)$-matrix with all entries equal
  to~$1$, and $\Id_{n}$ is the $n\times n$ identity matrix.
  It follows that $\Hess_{q}$ has rank~$4$.
  As the dimension of the $T^{3}$-orbit through $q$ is of dimension~$3$, it
  follows that the rank of~$\Hess_{q}$ is maximal.
  Since $q$ is a maximum or minimum for $\nu$, the equivariant Morse lemma then
  implies that there are coordinates $(y^{1},\dots,y^{4},z^{1},z^{2},z^{3})$
  on~$S^{7}$ near~$q$, so that
  $\nu(y,z) = \epsilon\sum_{j=1}^{4} (y^{i})^{2} = \epsilon\lVert y\rVert^{2}$.
  Thus the $T^{3}$-orbit through~$q$ has a tubular neighbourhood of the form
  $T^{3} \times \bR^{4}$, where
  $\nu|_{\bR^{4}} = \epsilon\lVert\cdot\rVert^{2}$.
  Thus for some small~$\delta$, each $s$ with
  $\lvert s\rvert \in I_{\delta} = (1/4-\delta,1/4)$ has that $\nu^{-1}(s)$ is
  equivariantly diffeomorphic to
  $T^{3} \times S^{3} \subset T^{3} \times \bR^{4}$, and $\nu^{-1}(s)/T^{3}$ is
  diffeomorphic to~$S^{3}$.
  As there are no critical values for $\nu$ in $(-1/4,0) \cup (0,1/4)$, the
  diffeomorphism type of $\nu^{-1}(t)$ and $\nu^{-1}(t)/T^{3}$ for
  $\lvert t\rvert \in (0,1/4)$ is the same as that for
  $\lvert s \rvert \in I_{\delta}$, and the result follows.
\end{proof}

\begin{remark}
  That the sets given here as the maximum and minimum of $\nu$ are the only
  $T^{3}$-invariant associative submanifolds of round~$S^{7}$ was already noted
  in \cite{lotay} as a direct consequence of~\cite{harvey-lawson}.
  So Proposition~\ref{prop:critical} already gives that there are no other
  non-zero critical values.
\end{remark}

\begin{remark}
  In \cite{kawai}, Kawai studies the seven-sphere with its squashed metric and
  $\G_{2}$-structure, and in particular $T^{3}$-invariant associative
  submanifolds in \cite[Proposition~6.1]{kawai}.
  Using the computations of that paper, and adjusting the choice of linear
  coordinates, one finds that the multi-moment map is a non-zero constant
  multiple of the function~$\nu$ given in~\eqref{eq:nu-S7}.
  Thus, after scaling, the results of Proposition~\ref{prop:round-S7}, also
  apply to the squashed metric, and the determined associative submanifolds
  agree.
\end{remark}

\subsection{Special classes}
\label{sec:special-classes}

As described in the introduction, there are several classes of nearly parallel
$\G_{2}$-geometries on seven-dimensional manifolds $M$, characterised by the
dimension of the space of Killing spinors~\cite{bar,friedrich-kms}.
Let us give a few concrete details, to clarify the symmetry considerations.
A general geometric reference is \cite{boyer-galicki}.

\smallbreak \textbf{$3$-Sasaki manifolds.}
The eight-dimensional metric cone $C = \bR \times M$, $g_{C} = dr^{2} + r^{2}g$ is
hyperK\"ahler, with closed two-forms
$\omega_{I}^{C},\omega_{J}^{C},\omega_{K}^{C}$, related by
$\omega_{A}^{C} = g_{C}(A\cdot{},{}\cdot{})$, for $A=I,J,K$, and with
$A^{2}=-\Id$ and $IJ = K = -JI$.
Writing $\iota\colon M \to \{1\} \times M \subset C$ for the inclusion, and
$E = r\partial_{r}$ for the Euler vector field on~$C$, we get one-forms
$\eta_{A} = \iota^{*}(E\hook \omega_{A}^{C})$ and two-forms
$\omega_{A} = \iota^{*}(\omega_{A}^{C})$, satisfying $d\eta_{A} = 2\omega_{A}$, for
$A=I,J,K$.
The vector fields dual to~$\eta_{A}$ are the restrictions of~$AE$, and generate a
locally free action of~$\SU(2)$ acting as rotation on
$\mathrm{Span}\{\eta_{I},\eta_{J},\eta_{K}\}$ and $\mathrm{Span}\{\omega_{I},\omega_{J},\omega_{K}\}$.

Bielawski \cite{bielawski} showed that the symmetry group of the $3$-Sasaki
$(M,g,\eta_{I},\eta_{J},\eta_{K})$ has rank at most~$2$, when $\eta_{A}$ are all preserved
individually and that all examples are obtained as $3$-Sasaki quotients of
higher-dimensional spheres.
Boyer \& Galicki \cite{boyer-galicki} gave explicit constructions showing that
are infinitely many compact seven-dimensional examples, and the cohomology of
these was computed by Hepworth \cite{hepworth}.
Note that the $T^{2}$-action in these examples commutes with the above action
of~$\SU(2)$.

Fixing the choice $A=I$, there is a
nearly parallel $\G_{2}$-structure with
\begin{equation*}
  \varphi = -\eta_{I}\wedge\omega_{I} + \eta_{J}\wedge\omega_{J} + \eta_{K}\wedge\omega_{K},\quad
  {\star}\varphi = \frac{1}{4}d\varphi = \frac{1}{2}(-\omega_{I}^{2} + \omega_{J}^{2} + \omega_{K}^{2}).
\end{equation*}
When the $3$-Sasaki structure has $T^{2}$-symmetry, this $\G_{2}$-structures
gets $T^{3} = T^{2} \times S^{1}$-symmetry, where $S^{1} \subset \SU(2)$ is the subgroup
preserving~$I$.
The action of $\SU(2)$ on~$\varphi$, now gives a two-dimensional family of nearly
parallel $\G_{2}$-structures with the same metric.

On the other hand Friedrich et al.\ \cite{friedrich-kms} showed there is another
$\G_{2}$-structure.
Regarding $\mathcal{H} = \ker \eta_{I} \cap \ker \eta_{J} \cap \ker \eta_{K}$ as a horizontal
distribution transverse to the $\SU(2)$-orbits, one may make a canonical
variation.
Explicitly, putting $\beta_{A} = \omega_{A} - \eta_{BC}$, for $(ABC) = (IJK)$, the
three-form
\begin{equation*}
  \tilde{\varphi} = aa^{2} \eta_{IJK} + ab^{2} \bigl(
  \eta_{I} \wedge \beta_{I}
  + \eta_{J} \wedge \beta_{J}
  + \eta_{K} \wedge \beta_{K}
  \bigr)
\end{equation*}
is a $\G_{2}$-form with metric
$\tilde{g} = a^{2}(\eta_{I}^{2} + \eta_{J}^{2} + \eta_{K}^{2}) + b^{2}g_{\mathcal{H}}$.
One may check that this $\G_{2}$-structure is nearly parallel only for
$a = 3/5$, $b=3/\sqrt{5}$, see also \cite{kawai}, indeed this is the squashed
metric of the previous section.
This time $\tilde{\varphi}$ is preserved by the $\SU(2)$-action, so if the
$3$-Sasaki manifold has $T^{2}$-symmetry, then the canonical variation
$\tilde{\varphi}$ has $T^{2} \times \SU(2)$-symmetry, and is preserved by any maximal
subgroup $T^{3}$-symmetry.
The cone on $\tilde{g}$ has full holonomy~$\Spin(7)$.

\smallbreak \textbf{Sasaki-Einstein manifolds.} Here the eight-dimensional
metric cone is Calabi-Yau, so carries a K\"ahler form $\omega_{C}$, with complex
structure~$I$, and a complex volume form $\Omega_{C}$.
Note that $S^{1}$ generated by $IE$, acts on $\Omega_{C}$ as multiplication
by~$e^{it}$.
The K\"ahler induces a one-form $\eta = \iota^{*}(E\hook\omega_{C})$ on~$M$, with
$d\eta = 2\omega$, for $\omega = \iota^{*}(\omega_{C})$.
One may regard $\omega$ as a Hermitian form on $\mathcal{H} = \ker \eta$.
Putting $\Omega = \iota^{*}(E \hook \Omega_{C})$ gives a $(3,0)$-form
on~$\mathcal{H}$ that satisfies $d\Omega = 4i \eta \wedge \Omega$.
The geometry on~$M$ is Sasaki-Einstein, in particular the induced metric is
Einstein with positive scalar curvature.
A nearly parallel $\G_{2}$-structure compatible with this metric is given by
\begin{equation*}
  \varphi = \eta \wedge \omega + \operatorname{Re} \Omega.
\end{equation*}
But there are circle of these structures, obtained by replacing $\Omega$ by
$e^{it}\Omega$, cf.\ \cite{fernandez-fkm}.

One may construct examples, where K\"ahler structure of the cone has
$T^{4}$-symmetry, by starting with methods of toric geometric geometry
\cite{martelli-sy,cho-fo,futaki-ow}.
Such a $T^{4}$ action, necessarily acts non-trivially on $\Omega_{C}$, and the
stabiliser gives a $T^{3}$-action that preserves the Calabi-Yau structure and
the nearly parallel $\G_{2}$-structure.
This construction includes the $3$-Sasaki examples as a special case.

\section{Torus-reduction}
\label{sec:torus-reduction}

Let $s$ be a non-zero regular value for~$\nu$.
We now wish to study the geometry of the principal bundle
$\nu^{-1}(s) \to \nu^{-1}(s)/T^3$.
We first write the nearly parallel $\G_{2}$-structure on $M$ in terms of the
invariant forms $d\nu$, $\theta_{i}$, $\alpha_{j}$.
Then we understand the geometry of the orbit space $\nu^{-1}(s)/T^3$ imposed by
the nearly parallel condition.

By inverting the system of identities from via \eqref{eq:nu-pqr-dnu} and
\eqref{eq:local-theta-alpha}, we get the expressions of the vectors
$e^1,\dots,e^7$ in terms of $d\nu, \theta_i, \alpha_j$, for all $i,j$.
One can then write the nearly parallel $\G_{2}$-structure
\eqref{eq:np-str1}--\eqref{eq:np-str2}, the induced metric, and the volume form
in terms of the invariant one-forms.  The results are the following:
\begin{align}
  \label{eqlist:np-metric}
  g &= \frac{1}{16\rho}d\nu^{2} + \theta^{T}H\theta + \frac{1}{\rho}\alpha^{T}H\alpha, \\
  \label{eqlist:np-vol}
  \vol & = \frac{h}{4\rho^{2}} \vol_{\theta}\wedge \vol_{\alpha} \wedge d\nu, \\
  \label{eqlist:np-str}
  \varphi
    &= \nu\,\vol_{\theta}
      - (\Hd{\theta})^{T} \wedge \alpha
      - \frac{\nu}{\rho} \theta^{T} \wedge \Hd{\alpha}
      - \frac{1}{4\rho} d\nu \wedge \theta^{T} \wedge H\alpha
      + \frac{1}{\rho} \vol_{\alpha},\\
  \label{eqlist:np-str-star}
  \star \varphi
    &= \frac{1}{4} d\nu \wedge \vol_{\theta}
      - \frac{1}{4} \,(\Hd{\theta})^{T} \wedge \tau
      + \frac{\nu}{4\rho} d\nu \wedge \,(\Hd{\theta})^{T} \wedge \alpha
      \eqbreak\nonumber
      - \frac{1}{4\rho} d\nu \wedge \theta^{T} \wedge \Hd{\alpha}
      - \frac{\nu}{4\rho^{2}} d\nu \wedge \vol_{\alpha}
      ,
\end{align}
where $h = \det H$, $\rho = h - \nu^{2}$, $\vol_{\theta} = \theta_{123}$,
$\vol_{\alpha} = \alpha_{123}$, $(\Hd{\gamma})_{i} = \gamma_{jk}$, for
$(ijk)=(123)$, and
\begin{equation}
  \label{eq:tau}
  \tau = \frac{4}{\rho} {\star}_{H\alpha}(H\alpha),
\end{equation}
so
\begin{equation*}
  \tau_{i} = \frac{4}{\rho} (H\alpha)_{j} \wedge (H\alpha)_{k}.
\end{equation*}
Note that $\tau$ can also be written two other forms
\begin{equation*}
  \tau = \frac{4}{\rho} (\adj H)\,\Hd{\alpha} = \frac{4h}{\rho} H^{-1} \Hd{\alpha},
\end{equation*}
where $\adj H$ is the adjugate of~$H$, so the matrix satisfying
$H \adj H = h \Id_{3}$.  It will be convenient to write
\begin{equation}
  \label{eq:sigma-beta}
  \sigma = H\tau = \frac{4h}{\rho}\Hd{\alpha}\quad\text{and}\quad\beta = H\alpha.
\end{equation}
On the other hand, differentiating the expressions for $\alpha_i$ in
\eqref{eq:horizontal-forms}, we obtain
\begin{equation*}
  d\alpha_i = d\nu \wedge \theta_i + \nu\, d\theta_i - 4\,{\star \varphi}(U_j,U_k,{}\cdot{},{}\cdot{}), \quad i=1,2,3,
\end{equation*}
and thus
\begin{equation}
  \label{eq:curvature}
  \nu\, d\theta = d\alpha - \tau + \frac{\nu}{\rho} d\nu \wedge \alpha.
\end{equation}
We also note
\begin{equation*}
  d\rho = dh - 2\nu\,d\nu.
\end{equation*}

We now wish to compare terms in $d\varphi=4\,{\star \varphi}$ and write an
equivalent formulation of this identity in terms of the invariant one-forms
$d\nu$, $\theta_i$, $\alpha_j$.
First, one can compute $d\varphi$ explicitly from \eqref{eqlist:np-str}:
\begin{equation*}
  \begin{split}
    d\varphi
    &= d\nu \wedge \vol_{\theta}
      + (\Hd{\theta})^{T} \wedge (\nu d\theta - d\alpha) \eqbreak
      + \theta^{T} \wedge (d\theta \barwedge \alpha  - \alpha \barwedge d\theta)
      + \theta^{T} \wedge d\biggl(\frac{\nu}{4h}\sigma\biggr)
      + \theta^{T} \wedge d\nu \wedge d\biggl(\frac{1}{4\rho}\beta\biggr)
      \eqbreak
      - \frac{\nu}{4h} d\theta^{T} \wedge \sigma
      + \frac{1}{4\rho} d\nu \wedge (d\theta)^{T} \wedge \beta
      - \frac{1}{\rho^{2}} dh \wedge \vol_{\alpha} + \frac{2\nu}{\rho^{2}} d\nu \wedge \vol_{\alpha}
      + \frac{1}{\rho} d\vol_{\alpha},
  \end{split}
\end{equation*}
where $(\gamma \barwedge \delta)_{i} = \gamma_{j} \wedge \delta_{k}$, for
$(ijk)=(123)$.
Using~\eqref{eq:curvature}, we see that the first two terms of $d\varphi$ agree
with the first three terms of $4\,{\star \varphi}$, so that
$d\varphi - 4\,{\star \varphi}$ only includes terms of degrees $0$ and $1$
in~$\theta$.  The terms of degree~$0$ give
\begin{equation}
  \label{eq:deg0}
  \begin{split}
    \frac{4}{\rho} dh \wedge \vol_{\alpha}
    &= d\nu \wedge (d\theta)^{T} \wedge \beta
      = \frac{1}{\nu} d\nu \wedge (d\alpha - \tau)^{T} \wedge \beta \\
    &= \frac{1}{\nu} d\nu \wedge d\alpha^{T} \wedge \beta - \frac{12h}{\nu\rho} d\nu \wedge \vol_{\alpha},
  \end{split}
\end{equation}
and those of degree $1$ yield
\begin{equation*}
  \begin{split}
    0
    &= d\theta \barwedge \alpha - \alpha \barwedge d\theta
      + d\biggl(\frac{\nu}{4h}\sigma\biggr)
      + d\nu \wedge d\biggl(\frac{1}{4\rho}\beta\biggr)
      - \frac{1}{4h} d\nu \wedge \sigma
    \\
    &= \frac{1}{\nu} d\,\Hd{\alpha}
      + \frac{1}{\nu}(\alpha \barwedge \tau - \tau \barwedge \alpha)
      + \frac{2}{\rho} d\nu \wedge \Hd{\alpha}
      + \frac{\nu}{4} d\biggl(\frac{1}{h}\sigma\biggr)
      + d\nu \wedge d\biggl(\frac{1}{4\rho}\beta\biggr)
    \\
    &= \frac{1}{4\nu} d\biggl(\frac{\rho}{h}\sigma\biggr)
      + \frac{1}{2h} d\nu \wedge \sigma
      + \frac{\nu}{4} d\biggl(\frac{1}{h}\sigma\biggr)
      + d\nu \wedge d\biggl(\frac{1}{4\rho}\beta\biggr)
    \\
    &= \frac{1}{4\nu}d\sigma
      + d\nu \wedge d\biggl(\frac{1}{4\rho}\beta\biggr)
      ,
  \end{split}
\end{equation*}
as $\alpha \barwedge \tau = \tau \barwedge \alpha$, by the symmetry of $H^{-1}$.
Thus
\begin{equation}
  \label{eq:d-sigma}
  d\sigma + \nu\,d\nu \wedge d\biggl(\frac{1}{\rho}\beta \biggr) = 0.
\end{equation}

Let $s$ be a regular value for~$\nu$.
We are now ready to study the geometry of the torus-bundle
$\nu^{-1}(s) \to \nu^{-1}(s)/T^3$.
The vertical space at each point is generated by $U_1$, $U_2$, $U_3$, so the
triple $\theta = (\theta_1,\theta_2,\theta_3)^{T}$ is a connection one-form of
the bundle.
One checks that $\mathcal{L}_{U_{\ell}}\alpha_i=0$ and $\alpha_i(U_{\ell})=0$
for all $\ell=1,2,3$.
So the one-forms $\alpha_i$, $i=1,2,3$, are basic and descend to a global
coframe on $\nu^{-1}(s)/T^3$.
Write $d_{3}$ for the exterior derivative on the base.
Note that $H$ is also $T^{3}$-invariant, so $H$, $\beta$, $\sigma$ and $\tau$
also descend to the base.
Pulling back \eqref{eq:d-sigma} to~$\nu^{-1}(s)$, gives
\begin{equation*}
  d_{3}\sigma = 0.
\end{equation*}
On the other hand, \eqref{eq:deg0} gives no relation on~$\nu^{-1}(s)/T^{3}$.

Restricting the identities in \eqref{eq:curvature} to $\nu^{-1}(s)$, and writing
$d_{6}$ for the exterior derivative on this submanifold, we get the curvature
two-form $d_{6}\theta$ on the three-manifold~$\nu^{-1}(s)/T^{3}$.
For $s \neq 0$, we have, omitting pull-backs,
\begin{equation}
  \label{eq:curvature-6}
  d_{6}\theta = \frac{1}{s}(d_{3}\alpha-\tau),
\end{equation}
so
\begin{equation*}
  d_{3}\tau = 0,
\end{equation*}
and the cohomology class $[d_{6}\theta]$ is that of $-\tau/s$.

We can thus formulate the following result.

\begin{proposition}
  \label{prop:tau-base}
  Let $s \neq 0$ be a regular value for $\nu$.
  On the three-manifold $\nu^{-1}(s)/T^{3}$, the two-form
  $\tau = (\tau_{1},\tau_{2},\tau_{3})^{T}$ of~\eqref{eq:tau} is closed with
  components linearly independent at each point.
  The cohomology class $[\tau/s]$ is integral and the two-form $\sigma=H\tau$
  of~\eqref{eq:sigma-beta} is also closed.
  As a principal $T^{3}$-bundle over the base, $\nu^{-1}(s)$ has curvature given
  by~\eqref{eq:curvature-6}.  \qed
\end{proposition}

\section{Inverse construction}
\label{sec:inverse-construction}

Let $Q$ be a three-manifold equipped with a coframe
$\alpha = (\alpha_{1},\alpha_{2},\alpha_{3})^{T}$ and a symmetric positive-definite matrix
$H = (H_{ij})_{i,j=1,2,3}$ of functions $H_{ij} \in C^{\infty}(Q)$.
In particular, $h = \det H>0$.
Assume that there is an $s \in \bR$ with $0 < s^{2} < h$; this is always the case
if $Q$ is compact, and write $\rho = h - s^{2}$.
We define $\tau$ by~\eqref{eq:tau} and assume that $\tau$ and $\sigma=H\tau$ are closed, and
that $[\tau/s]$ is integral in $H^{2}(Q,\bR^{3})$.
Here $\bR^{3}$ is the Lie algebra of $T^{3}$, and integrality is that elements
with values in the integral lattice in this Lie algebra, have integral periods.
We put $\beta = H\alpha$.

Let $\pi \colon E \to Q$ be a principal $T^{3}$-bundle with connection-one
form~$\theta$ satisfying~\eqref{eq:curvature-6}.
This bundle exists, since $[d\theta] = [\tau/s]$, which was assumed to be
integral.
We regard the total space $E$ as the level set $\nu^{-1}(s)$ of the previous
section.
Consider $M$ an open neighbourhood of $\{s\} \times E \subset \bR \times E$ that
is $\bR$-connected, meaning that $M \cap (\bR \times \{p\})$ is an open interval
for each~$p \in E$.
We construct a $\G_{2}$-structure on $M$ via the tensors
\eqref{eqlist:np-metric}--\eqref{eqlist:np-str-star}, with $\alpha,H,\theta$
being dependent on~$s$ and with $\nu$ replaced by $s$ throughout.
The nearly parallel condition is now equivalent to \eqref{eq:deg0} and
\eqref{eq:d-sigma} on $M$, again with $\nu$ replaced by~$s$.

\begin{remark}
  The $\G_{2}$-structure is determined solely by the
  three-form~\eqref{eqlist:np-str}.
  The set of $\G_{2}$ three-forms is open in $\Omega^{3}(M)$, so any small
  deformation of $\alpha$, $H$ and $\theta$ yields a new three-form, even if the
  deformation of $H$ is not symmetric.

  The metric and volume form are determined from the three-form
  by~\eqref{eq:g-vol}.
  If $H$ is not symmetric, then $d\nu$ is no longer orthogonal to~$\theta$, and
  thus $g$, $\vol$, and hence $\star\varphi$, are not given by the expressions
  \eqref{eqlist:np-metric}, \eqref{eqlist:np-vol} and
  \eqref{eqlist:np-str-star}.

  Thus the above statement that the nearly parallel condition is equivalent to
  \eqref{eq:deg0} and \eqref{eq:d-sigma} relies on $H$ being symmetric even when
  it varies with~$s$.
\end{remark}

The differential $d$ on~$M$ splits as the sum of the differential on $E$ and the
one in the remaining direction: we write $d=d_6+d_s$, where $d_6$ is the
differential on $E$ and $d_s\gamma = ds \wedge \gamma'$, with $\gamma'$ the
derivative of $\gamma$ with respect to~$s$, so
$\gamma' = \partial_{s} \hook d\gamma$.
Note that with this choice, we have
$(\gamma \wedge \delta)' = \gamma' \wedge \delta + \gamma \wedge \delta'$ and
$d_{6}(\gamma') = (d_{6}\gamma)'$, for any $s$-dependent differential forms
$\gamma$ and~$\delta$.
Also note that $\partial_{s} = (1/16\rho) d\nu^{\sharp}$ is not of constant
length in~$M$.

In this notation, \eqref{eq:curvature} gives
\begin{equation}
  \label{eq:theta-p}
  \theta' = \frac{1}{s}\alpha' + \frac{1}{\rho}\alpha,
\end{equation}
thus determining $\theta'$ from the flow of $\alpha$ and $H$.
Equation~\eqref{eq:deg0} implies
\begin{equation}
  \label{eq:h-p}
  h'\vol_{\alpha} = \frac{\rho}{4s} \beta^{T} \wedge d_{3}\alpha - \frac{3h}{s} \vol_{\alpha},
\end{equation}
and \eqref{eq:d-sigma} gives
\begin{equation}
  \label{eq:sigma-p}
  \sigma' = - s d_{3}\biggl(\frac{1}{\rho}\beta\biggr).
\end{equation}

Differentiating \eqref{eq:curvature} on~$M$, and substituting for $d\theta$, we
have
\begin{equation}
  \label{eq:tau-p}
  \tau' = \frac{1}{s} \tau - \frac{1}{s} d_{3}\biggl(\frac{h}{\rho}\alpha\biggr).
\end{equation}
Substituting into \eqref{eq:sigma-p}, using
$\sigma' = (H\tau)' = H'\tau + H\tau'$, gives
\begin{equation}
  \label{eq:H-p-tau}
  H'\tau = - \frac{1}{s} \sigma + \frac{1}{s} Hd_{3}\alpha - \frac{s}{\rho} d_{3}H \wedge \alpha.
\end{equation}

\begin{theorem}
  \label{thm:flow-construction}
  Let $(Q,h_{Q})$ be a smooth oriented Riemannian $3$-manifold with a global
  coframe $\alpha_{1},\alpha_{2},\alpha_{3}$.
  Write $h_{Q} = \alpha^{T}H\alpha$ with $H=(H_{ij})$, where
  $\alpha=(\alpha_{1},\alpha_{2},\alpha_{3})^{T}$.
  Assume there is an $s_{0} \in \bR$ with $0 < s_{0}^{2} < h = \det H$.
  For each cyclic permutation $(ijk)$ of $(123)$, put
  \begin{equation*}
    \sigma_{i} = \frac{4h}{\rho} \alpha_{j} \wedge \alpha_{k},\qquad
    \tau_{i} = \frac{4}{\rho} \beta_{j} \wedge \beta_{k},
  \end{equation*}
  where $\rho = h - s_{0}^{2}$ and
  $\beta_{p} = (H\alpha)_{p} = \sum_{a=1}^{3} H_{pa}\alpha_{a}$.

  Suppose for all $k=1,2,3$, that $\sigma_{k}$ and $\tau_{k}$ are closed and the
  cohomology class $[\tau/s_{0}] \in H^{2}(Q,\bR^{3})$,
  $\tau=(\tau_{1},\tau_{2},\tau_{3})^{T}$ is integral.
  Let $E \to Q$ be a principal $T^{3}$-bundle $E \to Q$ with connection one-form
  $(\theta_{1},\theta_{2},\theta_{3})$ such that $d_{6}\theta_{k} = (d\alpha_{k} - \tau_{k})/s_{0}$.

  Then there exists an open $\bR$-connected neighbourhood~$M$ of
  $\{s_{0}\} \times E$ in $\bR \times E$ on which there is a unique solution to
  the flow equations \eqref{eq:theta-p}, \eqref{eq:h-p}, \eqref{eq:sigma-p},
  \eqref{eq:tau-p}, \eqref{eq:H-p-tau} with the given initial data at
  $s = s_{0}$.
  Via \eqref{eqlist:np-metric}--\eqref{eqlist:np-str-star}, with $\nu$ the
  $\bR$-parameter, this solution defines the unique nearly parallel
  $\G_{2}$-structure on $M$ that is $T^{3}$-invariant and descends to the given
  data $(\alpha,h_{Q})$ on~$Q$ at $s=s_{0}$.
\end{theorem}

\begin{proof}
  From equations \eqref{eq:sigma-p} and~\eqref{eq:H-p-tau}, we can deduce a
  system of first order differential equations for the evolution of~$\alpha$
  and~$H$.
  Solutions to this smaller system will determine solutions to the flow
  equations.  To make this more explicit, write
  \begin{equation*}
    \alpha' = A\alpha,\quad d_{3}\alpha = B\,\Hd{\alpha},\quad
    d_{3}H_{ia} = \sum_{p=1}^{3} K_{ia}^{p}\alpha_{p},\quad d_{3}H \wedge \alpha =  R\,\Hd{\alpha},
  \end{equation*}
  for some $A, B, K^{p}, R \in C^{\infty}(Q,M_{3}(\bR))$.
  Note that $K_{ia}^{p} = K_{ai}^{p}$, by the symmetry of~$H$, and
  $R_{ia} = K_{ic}^{b} - K_{ib}^{c}$, for $(abc) = (123)$.

  Consider the equation \eqref{eq:H-p-tau}.  We have
  \begin{equation*}
    \frac{4h}{\rho}H' H^{-1} \Hd{\alpha}
    = \biggl(-\frac{4h}{s\rho}\Id_{3} + \frac{1}{s}HB - \frac{s}{\rho}R\biggr)\Hd{\alpha}.
  \end{equation*}
  Thus
  \begin{equation}
    \label{eq:H-p}
    H' = -\frac{1}{s}H + \frac{\rho}{4hs}HBH - \frac{s}{4h}RH
  \end{equation}
  is determined algebraically by the data~$(s,\alpha,H,d_{3}\alpha)$ on~$Q$.

  Before considering $\alpha'$, note that
  \begin{equation}
    \label{eq:h-H}
    3h = 3\det H
    = \sum_{(ijk)=(123)} \sum_{(abc)=(123)} H_{ia}H_{jb}H_{kc} - H_{ia}H_{jc}H_{kb},
  \end{equation}
  so $h' = \Tr(H'\adj H)$, where $(\adj H)_{ai} = H_{jb}H_{kc}-H_{jc}H_{kb}$,
  when $(ijk) = (123) = (abc)$.  Since $H \adj H = h \Id_{3}$, we conclude that
  \begin{equation}
    \label{eq:h-p-H-p-2}
    h' = -\frac{3h}{s} + \frac{\rho}{4s}\Tr(HB),
  \end{equation}
  because $\Tr R = 0$, by
  $\Tr R = \sum_{(ijk)=(123)} K_{ik}^{j} - K_{ij}^{k} = \sum_{(ijk)=(123)}
  K_{ki}^{j} - K_{ij}^{k} = 0$.

  Turning to $\alpha'$, first note that
  \begin{equation*}
    (\Hd{\alpha})' = (\alpha \barwedge \alpha)' = (A\alpha) \barwedge \alpha + \alpha \barwedge (A\alpha) = W\,\Hd{\alpha},
  \end{equation*}
  with $W = - A^{T} + (\Tr A) \Id_{3}$.  So
  \begin{equation}
    \label{eq:A-W}
    A = -W^{T} + \frac{1}{2}(\Tr W)\Id_{3}.
  \end{equation}
  But $\Hd{\alpha} = (\rho/4h)\sigma$, so equation~\eqref{eq:sigma-p} gives
  \begin{equation*}
    \begin{split}
      (\Hd{\alpha})'
      &= \biggl(\frac{\rho}{4h}\biggr)' \sigma
        - \frac{s\rho}{4h} d_{3}\biggl(\frac{1}{\rho} \beta\biggr) \\
      &= \frac{s}{h\rho}(sh' - 2h) \,\Hd{\alpha}
        + \frac{s}{4h\rho} d_{3}h \wedge H\alpha
        - \frac{s}{4h} d_{3}(H\alpha).
    \end{split}
  \end{equation*}
  From~\eqref{eq:h-H}, we have
  \begin{equation*}
    d_{3}h = \Tr((d_{3}H)\adj H) = \sum_{i,a,p=1}^{3}K_{ia}^{p} (\adj H)_{ia} \alpha_{p}
    \eqqcolon S^{T}\alpha.
  \end{equation*}
  Thus $d_{3}h \wedge H\alpha = F\,\Hd{\alpha}$ with
  \begin{equation}
    \label{eq:F-ip}
    F_{ip} = \sum_{n,a=1}^{3} (\adj H)_{na}\bigl(H_{ir}K_{na}^{q} - H_{iq}K_{na}^{r}\bigr)
    = H_{ir}S_{q} - H_{iq}S_{r},
  \end{equation}
  for $(pqr) = (123)$.  Note that
  \begin{equation*}
    \Tr F = \sum_{(ijk)=(123)} H_{ik}S_{j} - H_{ij}S_{k} = \sum_{(ijk)=(123)}
    H_{ki}S_{j} - H_{ij}S_{k} = 0.
  \end{equation*}
  We now have
  \begin{equation*}
    W = -\frac{5s}{\rho} \Id_{3} + \frac{s}{4h} \Tr(HB) \Id_{3}
    + \frac{s}{4h\rho}F
    - \frac{s}{4h}(R+HB)
  \end{equation*}
  with
  \begin{equation*}
    \Tr W = -\frac{15s}{\rho} + \frac{s}{2h}\Tr(HB).
  \end{equation*}
  Hence, from \eqref{eq:A-W}, we get
  \begin{equation}
    \label{eq:A}
    A = - \frac{5s}{2\rho} \Id_{3} - \frac{s}{4h\rho} F^{T}
    + \frac{s}{4h}(R^{T} + B^{T}H).
  \end{equation}
  Thus \eqref{eq:H-p}, $\alpha' = A\alpha$, with $A$ given by~\eqref{eq:A}, and
  $d_{3}$-derivatives of these equations, gives a first order system of
  differential equations for $H$, $\alpha$, $d_{3}\alpha$ and $d_{3}H$.
  The right-hand side of the system is smooth in this data away from $s=0$,
  $h=0$ and $\rho=0$, so is locally Lipschitz.

  We need to check consistency of the system.
  Note that \eqref{eq:tau-p} is a consequence of \eqref{eq:sigma-p}
  and~\eqref{eq:H-p-tau}.
  Also \eqref{eq:h-p} follows directly from~\eqref{eq:h-p-H-p-2}.
  Furthermore, \eqref{eq:sigma-p} and \eqref{eq:tau-p} show that the conditions
  $d_{3}\sigma = 0 = d_{3}\tau$ are preserved by the flow.

  It remains to see that $H'$ is symmetric.
  For $X \in C^{\infty}(Q,M_{3}(\bR))$, write
  $X^{\natural} \in C^{\infty}(Q,\bR^{3})$ for the vector with
  \begin{equation*}
    (X^{\natural})_{i} = X_{jk} - X_{kj} = (X-X^{T})_{jk}
  \end{equation*}
  where $(ijk) = (123)$.

  Note that for $\gamma \in \Omega^{1}(Q,\bR^{3})$ and
  $X,Y \in C^{\infty}(Q,M_{3}(\bR))$, we have
  \begin{equation*}
    \bigl((X\gamma) \barwedge (Y\Hd{\gamma})\bigr)_{i}
    = \sum_{a=1}^{3} X_{ja}\gamma_{a} \wedge \sum_{\mathclap{(pqr)=(123)}}Y_{kp}\gamma_{qr}
    = \sum_{a=1}^{3} X_{ja}Y_{ka}\vol_{\gamma},
  \end{equation*}
  when $(ijk) = (123)$.  Thus
  \begin{equation*}
    (Y\Hd{\gamma}) \barwedge (X\gamma) - (X\gamma) \barwedge (Y\Hd{\gamma}) = -(YX^{T})^{\natural} \,\vol_{\gamma}.
  \end{equation*}

  From $d_{3}\sigma = 0$, we have
  \begin{equation*}
    \begin{split}
      0
      &= \frac{\rho}{h}d_{3}\biggl(\frac{h}{\rho}\Hd{\alpha}\biggr)
        = d_{3}\alpha \barwedge \alpha - \alpha \barwedge d_{3}\alpha
        - \frac{s^{2}}{h\rho} S^{T}\alpha \wedge \Hd{\alpha} \\
      &= (B\Hd{\alpha}) \barwedge \alpha - \alpha \barwedge (B\Hd{\alpha}) - \frac{s^{2}}{h\rho} S \,\vol_{\alpha}.
    \end{split}
  \end{equation*}
  This gives
  \begin{equation}
    \label{eq:B-n}
    B^{\natural} = -\frac{s^{2}}{h\rho}S.
  \end{equation}
  Similarly $d_{3}\tau = 0$ gives
  \begin{equation*}
    \begin{split}
      0
      &= \rho\, d_{3}\biggl(\frac{1}{\rho}\Hd{\beta}\biggr)
        = (d_{3}\beta \barwedge \beta - \beta \barwedge d_{3}\beta) - \frac{1}{\rho} S^{T}\alpha \wedge (\adj H)\,\Hd{\alpha} \\
      &= d_{3}(H\alpha) \barwedge H\alpha - H\alpha \barwedge d_{3}(H\alpha)
        - \frac{1}{\rho} (\adj H)S \,\vol_{\alpha} \\
      &= R\Hd{\alpha} \barwedge H\alpha - H\alpha \barwedge R\Hd{\alpha}
        + (HB\Hd{\alpha}) \barwedge H\alpha - H\alpha \barwedge (HB\Hd{\alpha}) \eqbreak
        - \frac{1}{\rho} (\adj H)S \,\vol_{\alpha}
    \end{split}
  \end{equation*}
  and hence
  \begin{equation*}
    (RH + HBH)^{\natural} = -\frac{1}{\rho} (\adj H)S.
  \end{equation*}
  Now \eqref{eq:H-p} gives
  \begin{equation*}
    \begin{split}
      (H')^{\natural}
      &= -\frac{1}{s} H^{\natural} + \frac{\rho}{4hs}(HBH)^{\natural} - \frac{s}{4h}(RH)^{\natural} \\
      &= 0 + \frac{1}{4s}(HBH)^{\natural} + \frac{s}{4h\rho} (\adj H)S.
    \end{split}
  \end{equation*}
  But
  \begin{equation*}
    \begin{split}
      (HBH)^{\natural}_{i}
      &= \sum_{a,p=1}^{3} H_{ja}H_{kp}(B_{ap}-B_{pa}) \\
      &= \sum_{(abc)=(123)} B^{\natural}_{a} (H_{jb}H_{kc} - H_{jc}H_{kb})
        = (\adj H)_{ia}B^{\natural}_{a}.
    \end{split}
  \end{equation*}
  So
  \begin{equation*}
    (H')^{\natural} = \frac{1}{4s}(\adj H)B^{\natural} + \frac{s}{4h\rho}(\adj H)S
    = 0
  \end{equation*}
  by~\eqref{eq:B-n}, and $H'$ is symmetric.

  Thus we have a consistent system, and start data $(\alpha,H)$ at $s_{0}$
  determines a unique solution.
\end{proof}

\section{Examples and consequences}
\label{sec:exampl-consequences}

The above results are formulated with respect to fixed generators
$U_{1},U_{2},U_{3}$ for the $T^{3}$-action.
In general, a change of basis is given by a constant element of $\GL(3,\bR)$.

\begin{lemma}
  If we change the generating vector fields by $\tilde{U} = PU$, for
  $P \in \GL(3,\bR)$, then the quantities expressing the nearly
  parallel $\G_{2}$-structure change by
  \begin{gather*}
    \tilde{\nu} = (\det P)\nu,\quad
    \tilde{H} = PHP^{T},\quad
    \tilde{h} = (\det P)^{2}h,\quad
    \tilde{\rho} = (\det P)^{2}\rho,\\
    \tilde{\theta} = (P^{T})^{-1}\theta,\quad
    \tilde{\alpha} = (\adj P)^{T}\alpha,\quad
    \tilde{\sigma} = (\det P) P\sigma,\quad
    \tilde{\tau} = (\adj P)^{T}\tau.
  \end{gather*}
\end{lemma}

\begin{proof}
  The expression for $\tilde{\nu}$ follows directly from the definition.
  For $\tilde{H}$, we have
  $\tilde{H}_{ij} = \sum_{a,b=1}^{3} P_{ia}P_{jb}g(U_{a},U_{b})$ and the claimed
  formula.
  Now, writing $\tilde{\theta} = \tilde{P}\theta$, we have
  $\delta_{ij} = \tilde{\theta}_{i}(\tilde{U}_{j}) = \sum_{a,b=1}^{3}
  \tilde{P}_{ia}P_{jb}\theta_{a}(U_{b}) = \sum_{a=1}^{3} \tilde{P}_{ia}P_{ja}$,
  thus $\tilde{P} = (P^{T})^{-1}$ and
  $\tilde{\nu}\tilde{\theta} = (\adj P)^{T}\nu\theta$, from which the expression
  for~$\tilde{\alpha}$ follows.
  Note that $\tilde{\beta} = \tilde{H}\tilde{\alpha} = (\det P)P\beta$ and the
  transformation rule for $\tau = (4/\rho)\,\Hd{\beta}$ follows.
  Finally, we use $\sigma = H\tau$, to get~$\tilde{\sigma}$.
\end{proof}

As a consequence, we can always arrange that $H = \Id_{3}$ at a given point.
If $H = \Id_{3}$ globally on $Q$, then the two conditions $d_{3}\sigma = 0$ and
$d_{3}\tau = 0$ are the same.

\begin{proposition}
  \label{prop:eta}
  Let $Q$ be a connected oriented three-manifold, and suppose
  $\eta = (\eta_{1},\eta_{2},\eta_{3})^{T}$ be an $\bR^{3}$-valued closed
  two-form on~$Q$ whose components are linearly independent at each point, and
  whose cohomology class is integral.
  Then there is a nearly parallel $\G_{2}$-manifold~$M$ with $T^{3}$-symmetry,
  and a regular value $s_{0} \ne 0$ of the multi-moment map, such that either
  $\eta$ or $-\eta$ is the form $\tau/s_{0}$ on~$Q$ of
  Proposition~\ref{prop:tau-base}.
\end{proposition}

In fact, we will show that there is one such solution for each $s_{0}$ with
$s_{0}^{2} \in (0,1)$.

\begin{proof}
  We will generate initial data on~$Q$ as in
  Theorem~\ref{thm:flow-construction}, with $H = \Id_{3}$.
  Thus $h = 1$ and $0 < s_{0}^{2} < 1$. Then
  $\sigma = \tau = \epsilon s_{0}\eta$, $\epsilon \in \{\pm1\}$ and these are
  $d_{3}$-closed by assumption.
  Write $c_{0}^{2} = 1 - s_{0}^{2} = \rho$.
  We need to find a coframe~$\alpha$, so that
  $\sigma = (4/c_{0}^{2})\,\Hd{\alpha}$.
  This is equivalent to
  $\Hd{\alpha} = \epsilon(c_{0}^{2}s_{0}/4)\eta \eqqcolon \epsilon\hat{c} \eta$.

  As $\eta_{i}$, for $i=1,2,3$, is a non-zero two-form at each point, we have
  that each $\ker \eta_{i}$ is a one-dimensional distribution.
  Fixing a background metric $g_{Q}$ on $Q$, there are locally defined unit
  length vector fields $Y_{1},Y_{2},Y_{3}$ such that $Y_{i} \hook \eta_{i} = 0$,
  $i=1,2,3$.
  Indeed at each point there are two choices $\pm Y_{i}$.
  Locally, let $\gamma= (\gamma_{1},\gamma_{2},\gamma_{3})$ be the dual coframe
  to $(Y_{1},Y_{2},Y_{3})$.
  Then we may use the orientation of~$Q$ to fix the choice of sign for $Y_{i}$
  by requiring $\gamma_{i} \wedge \eta_{i} > 0$ for $i=1,2,3$.
  Hence $Y_{i}$ and $\gamma_{i}$ are globally defined.

  Now $Y_{i} \hook \eta_{i} = 0$ implies for $(ijk) = (123)$ that
  $\eta_{i} = f_{i}\gamma_{j} \wedge \gamma_{k}$ for some functions
  $f_{i} \in C^{\infty}(Q)$.
  Thus $\eta = f^{T}\,\Hd{\gamma}$.
  We wish to take $\alpha_{i} = a_{i}\gamma_{i}$ for some functions~$a_{i}$.
  We need to choose $a_{i}$ so that
  \begin{equation}
    \label{eq:f-aa}
    a_{j}a_{k} = \epsilon\hat{c}f_{i},\quad\text{for $(ijk) = (123)$.}
  \end{equation}
  Multiplying \eqref{eq:f-aa} by $a_{i}$, we get
  \begin{equation*}
    a_{i}a_{j}a_{k} = \epsilon\hat{c} a_{i}f_{i}.
  \end{equation*}
  But the left-hand side is independent of $(ijk)=(123)$, so $a_{i} = b/f_{i}$
  for some function~$b$ independent of~$i$.
  Putting this into \eqref{eq:f-aa}, we get
  $b^{2} = \epsilon\hat{c} f_{1}f_{2}f_{3}$.
  From linear independence of the $\eta_{i}$, we have $f_{1}f_{2}f_{3}$ is
  nowhere vanishing, so the product has a fixed sign, and we take $\epsilon$ to
  be this sign.
  Thus $\alpha_{i} = (\epsilon\hat{c} f_{1}f_{2}f_{3})^{1/2} \gamma_{i} / f_{i}$
  satisfies $\Hd{\alpha} = \hat{c} \epsilon\eta$, and we have the initial data
  needed to apply Theorem~\ref{thm:flow-construction}.
\end{proof}

\begin{remark}
  The sign change needed in the above proof can also be obtained by swapping
  $\eta_{2}$ and $\eta_{3}$.
\end{remark}

\begin{remark}
  Since $\sigma = H\tau$, the two triples of two-forms $\sigma$ and $\tau$
  determine~$H$.
  One can then formulate a version of Proposition~\ref{prop:eta} in terms of
  $s_{0}$, $\sigma$ and~$\tau$.
  However, the triples need to be restricted so that $H$ becomes symmetric and
  positive definite.
\end{remark}

\begin{example}
  Consider $Q=\bR^3$ as an Abelian Lie group with left-invariant data
  $(\alpha, H=(H_{ij}), s_{0} > 0)$ for the inverse construction.
  We have $d_{3}\alpha_{i} = 0$ and $h$ and $H$ are constant on~$Q$.
  We may change basis for the action so that $H(s_{0}) = \Id_{3}$, and then use
  $e = \alpha(s_{0})$ as a fixed basis for~$T^{*}_{0}Q \cong \bR^{3}$.
  Note that $0 < s_{0} < 1$.

  Consider possible evolution of the data in the form $H(s) = r(s)\Id_{3}$ and
  $\alpha(s) = u(s)e$ for some functions $r$ and $u$ of~$s$.
  Note that $d_{3}H = 0 = d_{3}\alpha$, so $B = 0 = K = R$.
  Equation~\eqref{eq:H-p} gives $r' = -r/s$, so $r = s_{0}/s$.
  We have from $\alpha' = A\alpha$ with $A = u'\Id_{3}$ given by~\eqref{eq:A},
  and $F$ by~\eqref{eq:F-ip} is zero, so
  \begin{equation*}
    u' = - \frac{5s}{2(r^{3} - s^{2})}u
    = - \frac{5s^{4}}{2(s_{0}^{3} - s^{5})}u,
  \end{equation*}
  so
  \begin{equation*}
    u = \frac{(s_{0}^{3} - s^{5})^{1/2}}{(s_{0}^{3} - s_{0}^{5})^{1/2}} =
    \frac{(s_{0}^{3} - s^{5})^{1/2}}{c_{0}s_{0}^{3/2}},
  \end{equation*}
  where $c_{0} = (1-s_{0})^{1/2} > 0$.
  One now checks that the flow equations are satisfied, so this is the unique
  solution with the given initial conditions, and all left-invariant solutions
  on $\bR^{3}$ are obtained in this way.

  The resulting nearly parallel $\G_{2}$-structure from
  \eqref{eqlist:np-metric}--\eqref{eqlist:np-str-star} has metric and three-form
  \begin{align*}
    g
    & = \frac{s^3}{16(s_{0}^3-s^5)}ds^2 + \frac{s_{0}}{s}\sum_{i=1}^3\theta_i^2
      + \frac{s^{2}}{c_{0}^{2}s_{0}^{2}}\sum_{i=1}^3(e^i)^2, \\
    \varphi
    & = s\theta_{123}
      - \frac{(s_{0}^3-s^5)^{1/2}}{c_{0}^{}s_{0}^{3/2}} \sum_{(ijk)=(123)} \theta_{ij} \wedge e^k
      - \frac{s^{4}}{c_{0}^{2}s_{0}^{3}} \sum_{(ijk)=(123)} \theta_i \wedge e^{jk} \eqbreak
      - \frac{s_{0}s^{2}}{4c_{0}^{}s_{0}^{3/2}
      (s_{0}^{3} - s^{5})^{1/2}} ds \wedge \sum_{k=1}^3 \theta_k \wedge e^k
      + \frac{s^{3}(s_{0}^3-s^5)^{1/2}}{c_{0}^{3}s_{0}^{9/2}}e^{123},
  \end{align*}
  with
  \begin{equation*}
    d_{6}\theta_{i} = -\frac{4}{c_{0}^{2}s_{0}^{2}} e^{jk}, \qquad
    \theta_{i}' = - \frac{3s^{3}}{2c_{0}s_{0}^{3/2}(s_{0}^{3} - s^{5})}e^{i}.
  \end{equation*}
  Extending the algebra by adding $p^{i}$, $i=1,2,3$, with $dp^{i} =
  d_{6}\theta_{i}$, we have
  \begin{equation*}
    \theta_{i} = p^{i} + x(s)e^{i},
  \end{equation*}
  where $x'e^{i} = \theta_{i}'$.
  This structure may thus be regarded as defined on $(0,s_{0}^{3/5}) \times G$, where
  $\mathfrak{g}$ is the nilpotent algebra $(0,0,0,23,31,12)$, and $G$ acts with
  cohomogeneity one.
\end{example}

Other first candidates for the inverse construction are homogeneous
manifolds~$Q$ of dimension three with left-invariant data.
As they have a global left-invariant coframe, such spaces~$Q$ are necessarily
Lie groups.
Bianchi's classification of three-dimensional Lie algebras then yields all
suitable Lie groups.
Our requirement of having a triple of linearly independent closed two-forms,
then implies that the Lie algebra is unimodular.
The only potential cases where one could get compact examples of nearly parallel
$\G_{2}$-structures via our inverse construction are for $\mathfrak{g}=\bR^{3}$ and
$\mathfrak{su}(2)$.
We have already investigated the first case above.
When the curvatures are integral we may get solutions over $Q=T^{3}$, but the
metrics are incomplete.
For $Q=\SU(2)$, one obtains nearly parallel $\G_{2}$-structures with a
cohomogeneity-one action of $\SU(2) \times T^{3}$, see a special case below.
This can not yield a complete metric: by Myers' Theorem, the seven-dimensional
manifold would be compact with finite fundamental group, which implies there are two special orbits, but at each special orbit at most one factor of~$T^{3}$
collapses, so the fundamental group remains infinite.

\begin{example}
  Let us consider the concrete case of $Q=SU(2)$.
  Let $e^{1},e^{2},e^{3}$ be the standard left-invariant coframe on~$SU(2)$ with
  $de^{i} = e^{jk}$.
  Take $\alpha = u(s)e$ and $H = r(s)\Id_{3}$.
  Then we have $d_{3}H = 0$ and $d_{3}\alpha = u(s)\,\Hd{e}$, so
  $B = (1/u(s))\Id_{3}$, and $u(s) A = u'(s)\Id_{3}$.
  Equations~\eqref{eq:H-p} and~\eqref{eq:A} give
  \begin{equation*}
    r' = - \frac{r}{s} + \frac{r^{3}-s^{2}}{4sru},\qquad
    u' = \frac{s}{4r^{2}} - \frac{5su}{2(r^{3}-s^{2})}.
  \end{equation*}
  Using the initial conditions $s_{0}=1/2$, $r(s_{0}) = 1 = u(s_{0})$, one can
  verify numerically that there is a smooth solution on $(0,s_{\max})$ for some
  $s_{\max} > s_{0}$ with $\lim_{s\uparrow s_{\max}}u(s) = 0$.
  The resulting nearly parallel $\G_{2}$-structure is
  \begin{align*}
    g &= \frac{1}{16(r^{3}-s^{2})} ds^{2} + r\sum_{i=1}^{3}\theta_{i}^{2}
        + \frac{u^{2}r}{r^{3}-s^{2}}\sum_{i=1}^{3} (e^{i})^{2},\\
    \varphi &= s\theta_{123}
        - u \sum_{(ijk)=(123)} \theta_{ij} \wedge e^{k}
        - \frac{su^{2}}{r^{3}-s^{2}} \sum_{(ijk)=(123)} \theta_{i} \wedge e^{jk}
        \eqbreak
        - \frac{ur}{4(r^{3}-s^{2})} ds \wedge \sum_{i=1}^{3} \theta_{i} \wedge e^{i}
        + \frac{u^{3}}{r^{3}-s^{2}} e^{123},
  \end{align*}
  with
  \begin{equation*}
    d_{6}\theta_{i} = \frac{u}{s}\biggl(1-\frac{4r^{2}u}{r^{3}-s^{2}}\biggr)e^{jk},\qquad
    \theta_{i}' = \biggl(\frac{1}{4r^{2}} - \frac{3u}{2(r^{3}-s^{2})}\biggr)e_{i}.
  \end{equation*}
  Noting that the $s$-derivative of the coefficient of $e^{jk}$ in $d_{6}\theta$ is
  indeed the coefficient of $e_{i}$ in $\theta_{i}'$, we have
  \begin{equation*}
    \theta_{i} = p^{i} + \frac{u}{s}\biggl(1-\frac{4r^{2}u}{r^{3}-s^{2}}\biggr)e^{i},
  \end{equation*}
  where $dp^{i} = 0$.
  Thus $p^{1},p^{2},p^{3},e^{1},e^{2},e^{3}$ can be taken as an invariant
  coframe for $T^{3} \times SU(2)$, and the (incomplete) nearly parallel
  $\G_{2}$-structure has symmetry group of rank~$4$.

  Note that this is not the general left-invariant solution on~$Q = \SU(2)$,
  unlike in the case of~$Q = \bR^{3}$; for general initial data, we can not
  simultaneously have $H = r \Id_{3}$ and $\alpha = u e$ at $s=s_{0}$.
\end{example}

\begin{corollary}
  There exist incomplete nearly parallel $\G_{2}$-structures that have
  $T^{4}$-symmetry.
\end{corollary}

Taking $Q$ compact with just a circle symmetry, and using circle invariant data
for the inverse construction, may yield examples with non-trivial stabilisers at
some points of $\nu^{-1}(0)$, hence potential compact solutions with
$T^{4}$-symmetry.


\begin{thebibliography}{99}

\bibitem{alexandrov} B.\ Alexandrov, \lq\lq On weak holonomy\rq\rq. \emph{Math.\
  Scand.\ } 96(2) (2005), 169--189.

\bibitem{bar} C.\ B\"ar, \lq\lq Real Killing Spinors and
  Holonomy\rq\rq. \emph{Commun.\ Math.\ Phys.\ } 154(3) (1993), 509--521.

\bibitem{ball-madnick} G.\ Ball, J.\ Madnick, \lq\lq Associative submanifolds of the
  Berger space\rq\rq. \emph{Comm.\ Anal.\ Geom.\ } 31(5) (2023), 1125--1175.

\bibitem{baum} H.\ Baum, Th.\ Friedrich, R.\ Grunewald, I.\ Kath, \lq\lq Twistors
  and Killing spinors on Riemannian manifolds\rq\rq. B. G. Teubner
  Verlagsgesellschaft, Stuttgart, Leipzig, 1991.

\bibitem{besse} A.\ L.\ Besse, \lq\lq Einstein manifolds\rq\rq. Classics in
  Mathematics. Springer-Verlag, Berlin, 2008.

\bibitem{bielawski} R.\ Bielawski, \lq\lq Complete hyper-K\"ahler $4n$-manifolds
  with a local tri-Hamiltonian $\mathbb{R}^n$-action\rq\rq.
  \emph{Math.\ Ann.\ } 314(3) (1999), 505--528.

\bibitem{boyer-galicki} C.\ P.\ Boyer, K.\ Galicki, \lq\lq Sasakian
  Geometry\rq\rq. Oxford Mathematical Monographs, Oxford, 2008.

\bibitem{bryant} R.\ L.\ Bryant, \lq\lq Metrics with exceptional
  holonomy\rq\rq. \emph{Ann.\ of Math.\ } 126(3) (1987), 525--576.

\bibitem{bryant1} R.\ L.\ Bryant, \lq\lq Some remarks on
  $\G_{2}$-structures\rq\rq.
  In \emph{Proceedings of $12$th G\"okova, Geometry-Topology Conference, 2005},
  G\"okova Geometry/Topology Conference (GGT), G\"okova, 2006, pp.\ 75--109.

\bibitem{butruille} J.-B.\ Butruille, \lq\lq Classification des vari\'et\'es
  approximativement k\"ahleriennes homog\`enes\rq\rq. \emph{Ann.\ Global Anal.\
  Geom.\ } 27(3) (2005), 201--225.

\bibitem{cho-fo} K.\ Cho, A.\ Futaki, H.\ Ono, \lq\lq Uniqueness and examples of
  compact toric Sasaki-Einstein metrics\rq\rq. \emph{Comm.\ Math.\ Phys.\ } 277(2)
  (2008), 439--458.

\bibitem{cleyton-swann} R.\ Cleyton, A.\ F.\ Swann, \lq\lq Cohomogeneity-one
  $\G_{2}$-structures\rq\rq. \emph{J.\ Geom.\ Phys.} 44(2--3) (2002), 202--220.

\bibitem{cleyton-swann-E} R.\ Cleyton, A.\ F.\ Swann, \lq\lq Einstein metrics via
  intrinsic or parallel torsion\rq\rq.
  \emph{Math.\ Z.\ } 247(3) (2004), 513--528.

\bibitem{cvetic} M.\ Cveti\v{c}, H.\ L\"u, D.-N.\ Page, C.\ N.\ Pope, \lq\lq New
  Einstein-Sasaki and Einstein spaces from Kerr-de Sitter\rq\rq. \emph{J.\ High
  Energy Phys.\ } 7 (2009), 082, 25~pages.

\bibitem{fernandez-fkm} M.\ Fern\'andez, A.\ Fino, A.\ Kovalev, V.\ Mu\~noz, \lq\lq
  Nearly parallel $\G_{2}$-manifolds: formality and associative
  submanifolds\rq\rq.  \emph{Math.\ Res.\ Lett.\ } 31(5) (2024), 1391--1434.

\bibitem{foscolo-haskins} L.\ Foscolo, M.\ Haskins, \lq\lq New $\G_{2}$-holonomy
  cones and exotic nearly K\"ahler structures on $S^6$ and
  $S^3 \times S^3$\rq\rq. \emph{Ann.\ of Math.\ } 185(1) (2017), 59--130.

\bibitem{friedrich-kms} Th.\ Friedrich, I.\ Kath, A.\ Moroianu, U.\ Semmelmann,
  \lq\lq On nearly parallel $\G_{2}$-structures\rq\rq. \emph{J.\ Geom.\ Phys.\ } 23(3--4)
  (1997), 259--286.

\bibitem{futaki-ow} A.\ Futaki, H.\ Ono, G.\ Wang, \lq\lq Transverse K\"ahler
  geometry of Sasaki manifolds and toric Sasaki-Einstein
  manifolds\rq\rq. \emph{J.\ Differential Geom.\ } 83(3) (2009), 585--635.

\bibitem{gray} A.\ Gray, \lq\lq Weak holonomy groups\rq\rq, \emph{Math.\ Z.\ } 123
  (1971), 290--300.

\bibitem{gray-nk} A.\ Gray, \lq\lq The structure of nearly K\"ahler
  manifolds\rq\rq. \emph{Math.\ Ann.\ } 223(3) (1976), 233--248.

\bibitem{harvey-lawson} F.\ R.\ Harvey, H.\ Blaine Lawson, Jr., \lq\lq Calibrated
  geometries\rq\rq, \emph{Acta Math.\ } 148 (1982), 47--157.

\bibitem{hepworth} R.\ A.\ Hepworth, \lq\lq The topology of certain
  {$3$}-Sasakian {$7$}-manifolds\rq\rq. \emph{Math.\ Ann.\ } 339(4) (2007),
  733--755.

\bibitem{kawai} K.\ Kawai, \lq\lq Some associative submanifolds of the squashed
  $7$-sphere\rq\rq. \emph{Quart.\ J.\ Math.\ } 66(3) (2015), 861--893.

\bibitem{kreck-stolz} M.\ Kreck, S. Stolz, \lq\lq Some nondiffeomorphic homeomorphic
  homogeneous $7$-manifolds with positive sectional curvature\rq\rq.
  \emph{J.\ Differential Geom.\ } 33(2) (1991), 465--486; correction 49(1)
  (1998), 203--204.

\bibitem{lotay} J.\ D.\ Lotay, \lq\lq Associative submanifolds of the
  $7$-sphere\rq\rq. \emph{Proc. Lond. Math. Soc.\ } 105(6) (2012), 1183--1214.

\bibitem{madsen-swann} T.\ B.\ Madsen, A.\ F.\ Swann, \lq\lq Multi-moment
  maps\rq\rq. \emph{Adv.\ Math.\ } 229(4) (2012) 2287--2309.

\bibitem{martelli-sy} D.\ Martelli, J.\ Sparks, S.-T.\ Yau, \lq\lq Sasaki-Einstein
  manifolds and volume minimisation\rq\rq. \emph{Comm.\ Math.\ Phys.\ } 280(3)
  (2008), 611--673.

\bibitem{mororianu-nagy} A.\ Mororianu, P.-A.\ Nagy, \lq\lq Toric nearly K\"ahler
  manifolds\rq\rq.  \emph{Ann.\ Global Anal.\ Geom.\ } 55(4) (2019), 703--717.

\bibitem{podesta} F.\ Podest\`a, \lq\lq Nearly parallel $\G_{2}$-structures with
  large symmetry group\rq\rq. \emph{Canad.\ J.\ Math.\ } 73(2) (2021), 339--359.

\bibitem{reidegeld} F.\ Reidegeld, \lq\lq Spaces admitting homogeneous
  $\G_2$-structures\rq\rq. \emph{Differential Geom.\ Appl.\ } 28(3) (2010),
  301--312.

\bibitem{russo2} G.\ Russo, \lq\lq Multi-moment maps on nearly K\"ahler
  six-manifolds\rq\rq. \emph{Geom.\ Dedicata} 213 (2021), 57--81.

\bibitem{russo} G.\ Russo, \lq\lq The Einstein condition on nearly K\"ahler
  six-manifolds\rq\rq. \emph{Expo.\ Math.\ } 39(3) (2021), 384--410.

\bibitem{russo-swann} G.\ Russo, A.\ F.\ Swann, \lq\lq Nearly K\"ahler six-manifolds
  with two-torus symmetry\rq\rq. \emph{J.\ Geom.\ Phys.\ } 138 (2019), 144--153.

\end{thebibliography}
\end{document}